\documentclass[10pt,twoside]{article}
\usepackage{mathrsfs}
\usepackage{amssymb}
\usepackage{amsmath}
\usepackage[all]{xy}
\usepackage{amsthm}

\usepackage{url}
\numberwithin{equation}{section}

\newtheorem{theorem}{Theorem}[section]
\newtheorem{proposition}[theorem]{Proposition}
\newtheorem{lemma}[theorem]{Lemma}

\newtheorem{corollary}[theorem]{Corollary}

\newtheorem{theorem*}{Theorem}

\theoremstyle{definition}
\newtheorem{definition}[theorem]{Definition}
\newtheorem{example}[theorem]{Example}

\theoremstyle{remark}

\usepackage{paralist}

\usepackage{indentfirst}
\newcommand{\Mod}{\operatorname{Mod}}
\newcommand{\Hom}{\operatorname{Hom}}
\newcommand{\Ext}{\operatorname{Ext}}
\newcommand{\A}{\mathscr{A}}

\newcommand{\id}{\operatorname{id}}
\newcommand{\ra}{\rightarrow}

\newcommand{\C}{\mathscr{C}}

\newcommand{\res}{\operatorname{res}}
\newcommand{\cores}{\operatorname{cores}}

\def\Im{\mathop{\rm Im}\nolimits}
\def\Coker{\mathop{\rm Coker}\nolimits}

\def\pd{\mathop{\rm pd}\nolimits}

\def\id{\mathop{\rm id}\nolimits}

\def\Mod{\mathop{\rm Mod}\nolimits}

\title{ \bf One-Sided Gorenstein Subcategories 
}
\vspace{0.2cm}

\author{Weiling Song$^{a}$,
Tiwei Zhao$^{b}$, Zhaoyong Huang$^{c,}$\thanks{{\it E-mail address}: songwl@njfu.edu.cn (W. L. Song), tiweizhao@qfnu.edu.cn (T. W. Zhao),
huangzy@nju.edu.cn (Z. Y. Huang)}\\
{\footnotesize \it $^a$Department of Applied Mathematics, College of Science, Nanjing Forestry University,}\\
{\footnotesize \it Nanjing 210037, Jiangsu Province, P.R. China;}\\
{\footnotesize \it $^b$School of Mathematical Sciences, Qufu Normal University, Qufu 273165, Shandong Province, P.R. China;}\\
{\footnotesize \it $^c$Department of Mathematics, Nanjing University, Nanjing 210093, Jiangsu Province, P.R. China}}
\date{ }
\begin{document}

\baselineskip=16pt
\maketitle

\begin{abstract}
We introduce the right (left) Gorenstein subcategory relative to
an additive subcategory $\C$ of an abelian category $\A$, and prove that the right Gorenstein subcategory
$r\mathcal{G}(\mathscr{C})$ is closed under extensions, kernels of epimorphisms, direct summands and finite direct sums.
When $\C$ is self-orthogonal,  we give a characterization for objects in $r\mathcal{G}(\mathscr{C})$,
and prove that any object in $\A$ with finite $r\mathcal{G}(\C)$-projective dimension
is isomorphic to a kernel (resp. a cokernel) of a morphism from an object in $\A$
with finite $\C$-projective dimension to an object in $r\mathcal{G}(\C)$. As an application,
we obtain a weak Auslander-Buchweitz context related to the kernel of a hereditary cotorsion pair in $\A$
having enough injectives.
\end{abstract}

\pagestyle{myheadings}
\markboth{\rightline {\scriptsize   W. L. Song, T. W. Zhao, Z. Y. Huang}}
         {\leftline{\scriptsize  One-Sided Gorenstein Subcategories}}


\section{Introduction} 

As a nice generalization of finitely generated projective modules,
Auslander and Bridger \cite{AB} introduced  finitely generated modules
having Gorenstein dimension zero over commutative Noetherian rings. For arbitrary
modules over general rings, Enochs and Jenda \cite{EJ95} introduced Gorenstein projective dimension,
which coincides with  Gorenstein dimension  for finitely generated modules
over commutative Noetherian rings; meanwhile, they also introduced Gorenstein injective modules
as the dual of Gorenstein projective modules. Since then, these modules have
become crucial research objects in Gorenstein homological algebra, and have
been studied extensively,
see \cite{AB,AM,C,CFH,CFrH,CI,EJ95,EJ00,Ho04}, and the references therein.

Let $\A$ be an abelian category and $\C$ an additive subcategory of $\A$. As a common generalization
of Gorenstein projective and Gorenstein injective modules, Sather-Wagstaff, Sharif and White \cite{SSW} introduced the
Gorenstein subcategory $\mathcal{G}(\C)$ of $\A$ relative to $\C$. It is shown that Gorenstein subcategories
share many nice properties of the categories of Gorenstein projective modules
and Gorenstein injective modules, see \cite{GD11}, \cite{Hu} and \cite{SSW}.
In \cite[Section 5]{SSW} the authors asked: do some important properties of the category
of Gorenstein projective (injective) modules
such as exactness, closure of kernels of epimorphisms and cokernels of monomorphisms,
hold true for Gorenstein subcategories? In fact, we will answer this question negatively.
Why are the answers negative? From the definition of the Gorenstein subcategory $\mathcal{G}(\C)$,
we know that $\C$ should be simultaneously a generator and a cogenerator for $\mathcal{G}(\C)$.
It leads to some limitations. The aim of this paper is to overcome such limitations by modifying
the definition of  Gorenstein subcategories. The paper is organized as follows.

In Section 2, we give some terminology and notations.

In Section 3, we introduce right (left) Gorenstein
subcategories relative to a subcategory $\C$ of an abelian category $\A$, such that when $\C$ is self-orthogonal,
the Gorenstein subcategory coincides with the intersection of the left and the right
Gorenstein subcategories. We prove that
the right Gorenstein subcategory $r\mathcal{G}(\C)$ is exact and closed under kernels of epimorphisms (Proposition \ref{extension-closed}),
however, the Gorenstein subcategory does not possess such properties in general (Example \ref{3.5}),
which answers \cite[Question 5.8]{SSW} negatively. Under the assumption that $\C$ is self-orthogonal,
we give some equivalent conditions for objects in $r\mathcal{G}(\mathscr{C})$ (Theorem \ref{3.6}), which shows that the subcategory
$r\mathcal{G}(\mathscr{C})$ has some kind of stability. Moreover, we
 prove that an object in $\A$ with finite $r\mathcal{G}(\C)$-projective dimension
is isomorphic to a kernel (resp. a cokernel) of a morphism from an object in $\A$
with finite $\C$-projective dimension to an object in $r\mathcal{G}(\C)$ (Theorem \ref{3.10}).

In Section 4, as an application of the above results, we prove that if a hereditary cotorsion pair $(\mathscr{U},\mathscr{V})$ with kernel $\C$
has enough injectives then
$(r\mathcal{G}(\C),\C$-$\pd^{<\infty},\C)$ is a weak Auslander-Buchweitz context, where
$\C$-$\pd^{<\infty}$ is the subcategory of $\A$ consisting of objects with finite $\C$-projective dimension
(Theorem \ref{4.8}).

\section{Preliminaries}

In this paper, $\A$ is an abelian category and all subcategories of $\A$ are full, additive and closed under isomorphisms.

For a subcategory $\mathscr{X}$ of $\A$, we write
$${^\perp{\mathscr{X}}}:=\{A\in\A\mid\operatorname{Ext}^{\geq 1}_{\A}(A,X)=0 \mbox{ for any}\ X\in \mathscr{X}\},\
{{\mathscr{X}}^\perp}:=\{A\in\A\mid\operatorname{Ext}^{\geq 1}_{\A}(X,A)=0 \mbox{ for any}\ X\in \mathscr{X}\},$$
$${^{\perp_1}{\mathscr{X}}}:=\{A\in\A\mid\operatorname{Ext}^{1}_{\A}(A,X)=0 \mbox{ for any}\ X\in \mathscr{X}\},\
{{\mathscr{X}}^{\perp_1}}:=\{A\in\A\mid\operatorname{Ext}^{1}_{\A}(X,A)=0 \mbox{ for any}\ X\in \mathscr{X}\}.$$
For subcategories $\mathscr{X},\mathscr{Y}$ of $\A$, we write $\mathscr{X}\perp\mathscr{Y}$ if
$\operatorname{Ext}^{\geq 1}_{\A}(X,Y)=0$ for any $X\in \mathscr{X}$ and $Y\in \mathscr{Y}$; and $\mathscr{X}$ is called
{\it self-orthogonal} if $\mathscr{X}\perp\mathscr{X}$.

Let $\C\subseteq \mathscr{X}$ be subcategories of $\A$. We say that $\C$ is a {\it generator} for $\mathscr{X}$
if for any $X\in \mathscr{X}$ there exists an exact sequence
$$0\ra X^{'} \ra C\ra X\ra 0$$ in $\A$ with $X^{'}\in \mathscr{X}$ and $C\in \C$. We say that $\C$ is a
{\it projective generator} for $\mathscr{X}$ if $\C$ is a generator for $\mathscr{X}$ and $\C\perp \mathscr{X}$.
Dually, we say that $\C$ is a {\it cogenerator} for $\mathscr{X}$
if for any $X\in \mathscr{X}$ there exists an exact sequence
$$0\ra X \ra C\ra X^{'}\ra 0$$ in $\A$ with $C\in \C$ and $X^{'}\in \mathscr{X}$. We say that $\C$ is an
{\it injective cogenerator} for $\mathscr{X}$ if $\C$ is a cogenerator for $\mathscr{X}$ and $\mathscr{X}\perp \C$.

Let $\mathscr{X}$ be a subcategory of $\A$ and $A\in\A$. The {\it $\mathscr{X}$-projective dimension}
$\mathscr{X}\mbox{-}\pd A$ of $A$ is defined to be the infimum integer $n$ for which  there exists an exact sequence
$$0\to X_n \to \cdots \to X_1\to X_0\to A\to 0$$
in $\A$ with all $X_i$ in $\mathscr{X}$, and we set $\mathscr{X}\mbox{-}\pd A=\infty$  if no such
integer exists. Dually, the {\it $\mathscr{X}$-injective dimension}
$\mathscr{X}\mbox{-}\id A$ of $A$ is defined to be the infimum integer $n$ for which  there exists an exact sequence
$$0\to A \to X^0 \to X^1 \to \cdots \to X^n\to 0$$
in $\A$ with all $X^i$ in $\mathscr{X}$, and we set $\mathscr{X}\mbox{-}\id A=\infty$  if no such
integer exists. We use $\mathscr{X}\mbox{-}\pd^{<\infty}$ and $\mathscr{X}\mbox{-}\id^{<\infty}$
to denote the subcategories of $\A$ consisting of objects with finite $\mathscr{X}$-projective and
$\mathscr{X}$-injective dimensions, respectively.

A sequence $\mathbb{E}$ in $\A$ is called {\it $\Hom_{\A}(\mathscr{X},-)$-exact}
(resp. {\it $\Hom_{\A}(-,\mathscr{X})$-exact}) if
$\Hom_{\A}(X,\mathbb{E})$ (resp. $\Hom_{\A}(\mathbb{E},X)$) is exact for any $X\in \mathscr{X}$.
Following \cite{SSW}, we write
$$
\res \widetilde{{\mathscr{X}}}:=\left\{A\in\A\mid  \begin{array}{c}
                                                \mbox{there exists an exact $\Hom_{\A}(\mathscr{X},-)$-exact sequence} \\
                                                \cdots\to X_i \to \cdots \to X_1 \to X_0 \to A \to 0 \mbox{ in $\A$ with all $X_i$ in $\mathscr{X}$}
                                              \end{array}
       \right\},
$$
and
$$
\cores \widetilde{{\mathscr{X}}}:=\left\{A\in\A\mid  \begin{array}{c}
                                                \mbox{there exists an exact $\Hom_{\A}(-,\mathscr{X})$-exact sequence} \\
                                                0\to A \to X^0 \to X^1 \to \cdots \to X^i \to \cdots \mbox{ in $\A$ with all $X_i$ in $\mathscr{X}$}
                                              \end{array}
       \right\}.
$$

\begin{definition}\label{2.1}  (\cite{SSW})
Let $\C$ be a subcategory of $\A$. The {\it Gorenstein subcategory}
$\mathcal{G}(\mathscr{C})$ of $\mathscr{A}$ relative to $\C$ is defined to be all
$G\in\mathscr{A}$ such that there exists an exact $\Hom_{\mathscr{A}}(\mathscr{C},-)$-exact and
$\Hom_{\mathscr{A}}(-,\mathscr{C})$-exact sequence
$$\cdots \to C_1 \to C_0 \to C^0 \to C^1 \to \cdots$$ in $\mathscr{A}$
with all $C_i,C^i$ in $\C$, such that $G\cong \Im(C_0\to C^0)$.
\end{definition}

The Gorenstein subcategory unifies the following notions: modules of
Gorenstein dimension zero (\cite{AB}), Gorenstein projective modules,
Gorenstein injective modules (\cite{EJ95}), $V$-Gorenstein projective
modules, $V$-Gorenstein injective modules (\cite{EJL05}),
$\mathscr{W}$-Gorenstein modules (\cite{GD11}), and so on; see \cite{Hu} for the
details.

\begin{definition} (\cite[Chapter 5.5]{Ro09})
A category $\mathscr{E}$ is called  {\it exact}  if $\mathscr{E}$ is a full subcategory of some abelian category $\mathscr{A}$
and it is closed under extensions.
\end{definition}

\section{One-sided Gorenstein Subcategories}

In this section, we fix a subcategory $\C$ of an abelian category $\A$. Following \cite[Lemma 5.7]{Hu},
if $\C\bot\C$ then
$$\mathcal{G}(\C)=({^\perp\mathscr{C}}\cap \cores\widetilde{\mathscr{C}})\cap
({\mathscr{C}^\perp}\cap \res\widetilde{\mathscr{C}}).$$
Motivated by this, we introduce right and left Gorenstein subcategories.

\begin{definition}\label{3.1}
We call
$$r\mathcal{G}(\C):={^\perp\mathscr{C}}\cap \cores\widetilde{\mathscr{C}}\
(resp.\ l\mathcal{G}(\C):={\mathscr{C}^\perp}\cap \res\widetilde{\mathscr{C}})$$ the {\it right}
(resp. {\it left}) {\it Gorenstein subcategory} of $\A$ relative to $\C$.
\end{definition}


In the following example, all rings are associative with identities and all modules are unitary.
For a ring $R$, $\Mod R$ is the category of left $R$-modules.

\begin{example}\label{3.2}
 \begin{enumerate}
\item[]
\item[(1)] In general, we have $r\mathcal{G}(\C)\neq l\mathcal{G}(\C)$. Let $R$ be a ring
and let $\A=\Mod R$. If $\C$ is the subcategory of $\A$ consisting of projective modules,
then $r\mathcal{G}(\C)$ is the subcategory of $\A$ consisting of Gorenstein projective modules and $l\mathcal{G}(\C)=\A$.
If $\C$ is the subcategory of $\A$ consisting of injective modules,
then $r\mathcal{G}(\C)=\A$ and $l\mathcal{G}(\C)$ is the subcategory of $\A$ consisting of Gorenstein injective modules.
\item[(2)] Let $R,S$ be rings and $_RC_S$ a semidualizing $(R,S)$-bimodule.
\begin{enumerate}
\item[(2.1)] If $\A=\Mod R$ and $\C$ is the subcategory of
$\A$ consisting of direct summands of direct sums of copies of $C$,
then $r\mathcal{G}(\C)$ is the subcategory of $\A$ consisting of $G_C$-projective modules by \cite[Proposition 2.4(1)]{LHX13},
and $l\mathcal{G}(\C)$ is the Bass class $\mathscr{B}_C(R)$ with respect to $C$ by \cite[Theorem 3.9]{TH1}.
\item[(2.2)] If $\A=\Mod S$ and $\C$ is the subcategory of
$\A$ consisting of direct summands of direct products of copies of $\Hom_\mathbb{Z}(C, \mathbb{Q}/\mathbb{Z})$,
where $\mathbb{Z}$ is the additive group of integers and $\mathbb{Q}$ is the additive group of rational numbers,
then $r\mathcal{G}(\C)$ is the Auslander class $\mathscr{A}_C(S)$ with respect to $C$ by \cite[Theorem 3.11(1)]{TH2}
and \cite[Proposition VI.5.1]{CE},
and $l\mathcal{G}(\C)$ is the subcategory of $\A$ consisting of $G_C$-injective modules by \cite[Proposition 2.4(2)]{LHX13}.
\end{enumerate}
\end{enumerate}
\end{example}

In the following, we study the properties of $r\mathcal{G}(\C)$. The dual versions of
all the results on $r\mathcal{G}(\C)$ also hold true  on $l\mathcal{G}(\C)$ by using completely dual arguments.


The following result shows that $r\mathcal{G}(\C)$ is an exact category, and has a quasi-resolving structure, that is closed under kernels of  epimorphisms..

\begin{proposition}\label{extension-closed}
\begin{enumerate}
\item[]
\item[(1)] The subcategory $r\mathcal{G}(\C)$ is closed under extensions,
 direct summands and finite direct sums.
\item[(2)] The subcategory $r\mathcal{G}(\C)$ is closed under
kernels of epimorphisms.
\end{enumerate}
\end{proposition}

\begin{proof}
(1) Let
$$0\rightarrow L\rightarrow M\rightarrow N\rightarrow 0$$ be an exact sequence in $\mathscr{A}$
with $L,N\in r\mathcal{G}(\C)$. Then it is $\Hom_\A(-,\C)$-exact.
Since $L,N\in r\mathcal{G}(\C)$, then $M\in{^\bot\C}$ and there exist exact $\Hom_{\A}(-,\C)$-exact sequences
\begin{gather*}
    \xymatrix{
0\ar[r] &L \ar[r]^{{d^{'}}^0} & {C^{'}}^0 \ar[r]^{{d^{'}}^1} &{C^{'}}^1 \ar[r]
& \cdots}\ {\rm and} \ \
\xymatrix{
0 \ar[r] & N \ar[r]^{{d^{''}}^0} &{C^{''}}^0 \ar[r]^{{d^{''}}^1}
& {C^{''}}^1 \ar[r]  & \cdots}
\end{gather*}
in $\A$ with all ${C^{'}}^i$ and ${C^{''}}^i$ in $\C$. It is trivial that all cokernels of
${d^{'}}^i$ and ${d^{''}}^i$ are in $\cores\widetilde{\mathscr{C}}$.
Put $C^0:={{C^{'}}}^0\oplus {C^{''}}^0$ and consider the following diagram
$$\xymatrix@R=20pt@C=20pt{&0\ar[d]&&0\ar[d]&\\
0\ar[r]&L\ar[r]^{f}\ar[d]^{{d^{'}}^0}&M\ar[r]^{g}&N\ar[r]\ar[d]^{{d^{''}}^0}&0\\
0\ar[r]&{C^{'}}^0\ar[r]^{\!\!{1\choose 0}}&C^0\ar[r]^{\ (0\ 1)}&{C^{''}}^0\ar[r]&0.}
$$
Since the upper row in this diagram is $\Hom_\A(-,\C)$-exact, we get an epimorphism
$\mbox{Hom}_{\A}(M,{{C^{'}}}^0) \twoheadrightarrow \mbox{Hom}_{\A}(L,{{C^{'}}}^0)$ and
there exists $\alpha:M\rightarrow {{C^{'}}}^0$ such that ${d^{'}}^0=\alpha f$.
Putting $d^0:={\alpha \choose {{d^{''}}^0 \!g}}$,
then we obtain the following commutative diagram with exact columns and rows
$$
 \xymatrix@R=20pt@C=20pt{
 &0\ar[d]&0\ar@{-->}[d]&0\ar[d]&\\
 0\ar[r]&L\ar[r]^{f}\ar[d]^{{d'}^0}&M\ar@{.>}[ld]_\alpha\ar[r]^{g}\ar@{-->}[d]^{d^0}
 &N\ar[r]\ar[d]^{{d^{''}}^0}&0\\
 0\ar[r]&{C^{'}}^0\ar[r]^{\!\!{1\choose 0}}\ar[d]&C^0
 \ar[r]^{\ (0\ 1)}\ar@{-->}[d]&{C^{''}}^0\ar[r]\ar[d]&0\\
0\ar@{-->}[r]&\mbox{Coker}{d^{'}}^0\ar@{-->}[r]\ar[d]&\mbox{Coker}d^0\ar@{-->}[r]\ar@{-->}[d]
&\mbox{Coker}{d^{''}}^0\ar@{-->}[r]\ar[d]&0\\
&0&0&0.&}
$$
It is easy to check that the bottom row and the middle column
 in this diagram are $\mbox{Hom}_{\A}(-,\C)$-exact.
Repeating this process, we may get an exact $\mbox{Hom}_{\A}(-,\C)$-exact sequence
$$\xymatrix@C=0.5cm{
0 \ar[r] &M\ar[r] & C^0 \ar[r]^{} &C^1 \ar[r]^{} & \cdots}$$ in $\A$ with all $C^i$ in $\C$. Thus
$M\in\cores\widetilde{\mathscr{C}}$ and $M\in r\mathcal{G}(\C)$.

It is trivial that ${^\perp\mathscr{C}}$ is closed under direct summands. By \cite[Theorem 4.6(1)]{Hu},
$\cores\widetilde{\mathscr{C}}$ is closed under direct summands. Thus $r\mathcal{G}(\C)$
is also closed under direct summands.

It is easy to see that $r\mathcal{G}(\C)$ is closed under finite direct sums.

(2) Consider an exact sequence
$$0\to A\to A_1\to A_2\to 0$$ in $\A$ with $A_1,A_2\in r\mathcal{G}(\C)$. First, it is $\mbox{Hom}_{\A}(-,\C)$-exact and $A\in{^\perp\C}$. Moreover,
there exist  exact $\Hom_\A(-,\C)$-exact sequences
$$0\to  A_1\to C_1^0\to C_1^1\to\cdots\to C_1^n\to \cdots,$$
$$0\to A_2\to C_2^0\to C_2^1\to\cdots\to C_2^n\to \cdots$$
in $\A$ with all $C^i_m$ in $\C$. By \cite[Theorem 3.8(1)(5)]{Hu}, we get an exact  $\mbox{Hom}_{\A}(-,\C)$-exact sequence
$$0\to A\to C_1^0\to C_1^1\oplus C_2^0\to \cdots \to C_1^n\oplus C_1^{n-1}\to \cdots,$$
that is, $A\in\cores\widetilde{\mathscr{C}}$.
Thus $A\in r\mathcal{G}(\C)$, as desired.
\end{proof}

Sather-Wagstaff, Sharif and White \cite[Question 5.8]{SSW} posed  the following questions:
Is $\mathcal{G}(\mathscr{C})$ always exact? Is $\mathcal{G}(\mathscr{C})$ always closed under kernels
of epimorphisms or cokernels of monomorphisms? The following example answers these two questions
negatively.

\begin{example}\label{3.5}
Let $\Lambda=kQ$ be a finite-dimensional hereditary path algebra over an algebraically closed field $k$, where $Q$ is the quiver
$$1\leftarrow 2\leftarrow 3.$$ Then the Auslander-Reiten quiver of $\Lambda$ is as follows (\cite[Chapter IV, Example 4.10]{ASS}):
$$\xymatrix@R=5pt@C=10pt{
[S_1]\ar[dr]\ar@{.}[rr]&&[S_2]\ar@{.}[rr]\ar[dr]&&[S_3]\\
&[P_2]\ar@{.}[rr]\ar[ur]\ar[dr]&&[I_2]\ar[ur]&\\
&&[P_3],\ar[ur]&&
}$$
where the symbol $[M]$ denotes the isomorphism class of a module $M$, and
\begin{itemize}
  \item[(1)] $S_1=(k\leftarrow 0\leftarrow 0)$, $S_2=(0\leftarrow k\leftarrow 0)$ and $S_3=(0\leftarrow 0\leftarrow k)$ are all simple modules;
  \item[(2)] $P_1=S_1$, $P_2=(k\leftarrow k\leftarrow 0)$ and $P_3=(k\leftarrow k\leftarrow k)$ are all indecomposable projective modules;
  \item[(3)]  $I_1=P_3$, $I_2=(0\leftarrow k\leftarrow k)$ and $I_3=S_3$ are all indecomposable injective modules.
\end{itemize}
Let $\A$ be the category of finitely generated left $\Lambda$-modules.
\begin{itemize}
  \item[(i)] Let
$\mathscr{C}=\mbox{add}(S_1\oplus S_2\oplus P_3\oplus S_3)$ be the full subcategory of $\A$ consisting of
all direct summands of finite direct sum of $S_1\oplus S_2\oplus P_3\oplus S_3$. By computation,
$\mathcal{G}(\mathscr{C})=\mbox{add}(S_1\oplus S_2\oplus P_3\oplus S_3)$. From the exact sequence
$$0\ra S_1\ra P_2\ra S_2\ra 0,$$ we see that $\mathcal{G}(\mathscr{C})$ is not exact;
and from the following exact sequences
$$0\ra P_2\ra P_3\ra S_3\ra 0\ {\rm and} \ 0\ra S_1\ra P_3\ra I_2\ra 0,$$
we see that $\mathcal{G}(\mathscr{C})$ is neither closed under kernels of epimorphisms nor
closed under cokernels of monomorphisms.
\item[(ii)]Let
$\mathscr{C}=\mbox{add}(S_1\oplus P_3\oplus S_3)$. In this case, $\mathscr{C}\perp \mathscr{C}$.
To illustrate it, we only need to check  $\Ext_{\Lambda}^1(S_3,S_1)=0$. In fact, by \cite[Chapter IV,  Corollary 2.14]{ASS} we have
$$\Ext_{\Lambda}^1(S_3,S_1)\cong D\Hom_{\Lambda}(S_1,\tau S_3)\cong D\Hom_{\Lambda}(S_1,S_2)=0,$$ where $D:=\Hom_k(-,k)$. Moreover,  by computation,
$\mathcal{G}(\mathscr{C})=\mbox{add}(S_1\oplus P_3\oplus S_3)$. Thus the second and third short exact sequences in (i) show that
$\mathcal{G}(\mathscr{C})$ is still neither closed under kernels of epimorphisms nor
closed under cokernels of monomorphisms.
\end{itemize}
\end{example}

In the rest of this section, we always assume $\mathscr{C}\perp \mathscr{C}$. In this setting, $\C\subseteq r\mathcal{G}(\C)$.  We have the following lemma.

\begin{lemma}\label{3.3}
For an object $L\in\A$, $L\in r\mathcal{G}(\C)$ if and only if there exists an exact ($\Hom_\A(-,\C)$-exact) sequence
$$0\rightarrow L\rightarrow C\rightarrow N\rightarrow 0$$ in $\A$ with $C\in\C$ and $N\in r\mathcal{G}(\C)$.
\end{lemma}

\begin{proof}

The necessity follows from Proposition \ref{extension-closed}(2).

For the sufficiency, since $L\in r\mathcal{G}(\C)$, there exists an exact $\Hom_{\A}(-,\mathscr{C})$-exact sequence
$$0\rightarrow L\rightarrow C\rightarrow N\rightarrow 0$$
in $\A$ with $C\in\C$ and $N\in \cores\widetilde{\mathscr{C}}$. Since $\mathscr{C}\perp \mathscr{C}$,
we have $N\in{^\perp\mathscr{C}}$ and thus $N\in r\mathcal{G}(\C)$.
\end{proof}

\begin{proposition}\label{3.4}
If $$0\rightarrow L\rightarrow M\rightarrow N\rightarrow 0$$
is an exact sequence in $\mathscr{A}$ with $L,M\in r\mathcal{G}(\C)$, then
$N\in r\mathcal{G}(\C)$ if and only if $N\in{^{\perp_1}\C}$.
\end{proposition}

\begin{proof}

The necessity is trivial. In the following we prove the sufficiency.
Let $$0\rightarrow L\rightarrow M\rightarrow N\rightarrow 0$$
be an exact sequence in $\mathscr{A}$ with $L,M\in r\mathcal{G}(\C)$ and
$\Ext^1_{\mathscr{A}}(N,C)=0$ for any $C\in \mathscr{C}$.
Since $L\in r\mathcal{G}(\C)$, by Lemma \ref{3.3} there exists an exact sequence
$$0\rightarrow L\rightarrow C\rightarrow L^{'}\rightarrow0$$ in $\A$ with $C\in\C$ and
$L^{'}\in r\mathcal{G}(\C)$. Consider the following push-out diagram
$$\xymatrix@R=20pt@C=20pt{
&0\ar[d]&0\ar@{-->}[d]&&\\
0\ar[r]&L\ar[r]\ar[d]&M\ar[r]\ar@{-->}[d]&N\ar@{==}[d]\ar[r]&0\\
0\ar@{-->}[r]&C\ar@{-->}[r]\ar[d]&T\ar@{-->}[r]\ar@{-->}[d]&N\ar@{-->}[r]&0\\
&L^{'}\ar@{==}[r]\ar[d]&L^{'}\ar@{-->}[d]&&\\
&0&0.&& }$$
By Proposition \ref{extension-closed}(1),  the middle column in this diagram shows that $T\in r\mathcal{G}(\C)$.
Since $\Ext^1_{\A}(N,C)=0$ by assumption, the middle row in the above diagram splits. Thus $N$ is isomorphic to a direct summand of $T$,
and hence it is in $r\mathcal{G}(\C)$ by  Proposition \ref{extension-closed}(1).
\end{proof}

In the following, we give some equivalent conditions for objects in $r\mathcal{G}(\C)$.

\begin{theorem}\label{3.6}
For any $M\in\A$, the following statements are equivalent.
\begin{enumerate}
\item[(1)] $M\in r\mathcal{G}(\C)$.
\item[(2)] $M\in {^\perp\mathscr{C}}$ and for any subcategory $\mathscr{D}$ of $\mathscr{A}$ such that
$\C$ is an injective cogenerator for $\mathscr{D}$, there exists an exact $\Hom_{\A}(-,\C)$-exact sequence
$$0\ra M\ra D^0 \rightarrow D^1\rightarrow \cdots$$ in $\A$ with all $D^i$ in $\mathscr{D}$.
\item[(3)] There exists an exact $\Hom_{\A}(-,\C)$-exact sequence $${\mathbb{G}}:=\xymatrix@C=0.5cm{
  \cdots \ar[r] & G_1 \ar[r]^{} & G_0 \ar[r]^{} & G_{-1} \ar[r]^{} & G_{-2} \ar[r] & \cdots }$$
in $\A$ with all $G_i$ in $r\mathcal{G}(\C)$ such that
$M\cong\operatorname{Ker}(G_{-1}\rightarrow G_{-2})$.
\item[(4)] For any subcategory $\mathscr{D}$ of $\mathscr{A}$ with
$\mathscr{C}\subseteq\mathscr{D}\subseteq r\mathcal{G}(\C)$,
there exists an exact $\Hom_{\A}(-,\mathscr{D})$-exact sequence $${\mathbb{G}}:=\xymatrix@C=0.5cm{
  \cdots \ar[r] & G_1 \ar[r]^{} & G_0 \ar[r]^{} & G_{-1} \ar[r]^{} & G_{-2} \ar[r] & \cdots }$$
in $\A$ with all $G_i$ in $r\mathcal{G}(\C)$ such that
$M\cong\operatorname{Ker}(G_{-1}\rightarrow G_{-2})$.
\item[(5)] There exists an exact $\Hom_{\A}(-,r\mathcal{G}(\C))$-exact
sequence $${\mathbb{G}}:=\xymatrix@C=0.5cm{
  \cdots \ar[r] & G_1 \ar[r]^{} & G_0 \ar[r]^{} & G_{-1} \ar[r]^{} & G_{-2} \ar[r] & \cdots }$$
in $\A$ with all $G_i$ in $r\mathcal{G}(\C)$ such that
$M\cong\operatorname{Ker}(G_{-1}\rightarrow G_{-2})$.
\end{enumerate}
\end{theorem}

\begin{proof}
The implications $(1)\Rightarrow (2)$ and $(5)\Rightarrow(4)\Rightarrow(3)$ are trivial.

$(2)\Rightarrow (1)$ Let $M\in {^\perp\mathscr{C}}$ and
$$0\ra M\ra D^0 \rightarrow D^1\rightarrow \cdots$$
be an exact $\Hom_{\A}(-,\C)$-exact sequence in $\A$ with all $D^i$ in $\mathscr{D}$. Put $M^1:=\Im(D^0 \rightarrow D^1)$.
Since $\C$ is an injective cogenerator for $\mathscr{D}$ by assumption, $M^1\in{^\perp\mathscr{C}}$ and
there exists an exact sequence
$$0\ra D^0\ra C^0\ra {D^{'}}\ra 0$$ in $\A$ with $C^0\in \C$ and $D^{'}\in \mathscr{D}$.
Consider the following push-out diagram:
$$\xymatrix@R=20pt@C=20pt{& & 0 \ar[d] & 0 \ar@{-->}[d]& &\\
0 \ar[r] & M \ar@{==}[d] \ar[r] & D^0 \ar[d] \ar[r] &M^1 \ar@{-->}[d] \ar[r] & 0\\
0 \ar@{-->}[r] & M \ar@{-->}[r] & C^0 \ar@{-->}[r] \ar[d] & M^{'} \ar@{-->}[d] \ar@{-->}[r] & 0 &\\
& & D^{'} \ar@{==}[r] \ar[d] & D^{'} \ar@{-->}[d]& &\\
& & 0 & 0. & & }$$
Since ${D^{'}},M^1\in {^{\perp}\mathscr{C}}$, we have that $M^{'}\in{^{\perp}\mathscr{C}}$ and therefore the middle row
 is $\Hom_{\A}(-,\C)$-exact. Similarly, we get an exact $\Hom_{\A}(-,\C)$-exact sequence
$$0\rightarrow M^{'}\rightarrow C^1\rightarrow M^{''}\rightarrow 0$$ in $\A$ with $M^{''}\in {^{\perp}\mathscr{C}}$.
Continuing this process, we obtain an exact $\Hom_{\A}(-,\C)$-exact sequence
$$0\ra M \rightarrow C^0\rightarrow C^1\rightarrow  \cdots$$
in $\A$ with all $C^i$ in $\C$, as desired.

$(1)\Rightarrow (5)$ It holds by setting
$${\mathbb{G}}:=\xymatrix@C=0.5cm{
 \cdots \ar[r] & 0 \ar[r]^{} & M \ar[r]^{\operatorname{id}} & M \ar[r]^{} & 0 \ar[r] & \cdots }.$$

$(3)\Rightarrow (1)$ Let
$${\mathbb{G}}:=\xymatrix@C=0.5cm{
  \cdots \ar[r] & G_1 \ar[r]^{} & G_0 \ar[r]^{} & G_{-1} \ar[r]^{} & G_{-2} \ar[r] & \cdots }$$
be an exact $\Hom_{\A}(-,\C)$-exact sequence in $\A$ with all $G_i$ in $r\mathcal{G}(\C)$
and $M\cong\mbox{Ker}(G_{-1}\rightarrow G_{-2})$. It is easy to see that $M_{-i}\in{^{\perp}\C}$,
where $M_{-i}:=\mbox{Ker}(G_{-(i+1)}\rightarrow G_{-(i+2)})$ for any $i\geq 0$ (note: $M_0=M$).

Since $G_{-1}\in r\mathcal{G}(\C)$, by Lemma \ref{3.3} there exists an exact $\Hom_{\A}(-,\C)$-exact sequence
$$0\rightarrow G_{-1}\rightarrow C_{-1}\rightarrow G_{-1}^{'}\rightarrow 0$$ in $\A$ with
$C_{-1}\in {\C}$ and $G_{-1}^{'}\in r\mathcal{G}(\C)$. Consider the following push-out diagram
 $$\xymatrix@R=20pt@C=20pt{& & 0 \ar[d] & 0 \ar@{-->}[d]& &\\
0 \ar[r] & M \ar@{==}[d] \ar[r] & G_{-1} \ar[d] \ar[r] &M_{-1} \ar@{-->}[d] \ar[r] & 0\\
0 \ar@{-->}[r] & M \ar@{-->}[r] & C_{-1} \ar@{-->}[r] \ar[d] & Q_{-1} \ar@{-->}[d] \ar@{-->}[r] & 0 &\\
& & G_{-1}^{'} \ar@{==}[r] \ar[d] & G_{-1}^{'} \ar@{-->}[d]& &\\
& & 0 & 0. & & }$$
Then $Q_{-1}\in{^{\perp}\C}$ and the middle row in this diagram is $\Hom_{\A}(-,\C)$-exact.
Now consider the following push-out diagram
$$\xymatrix@R=20pt@C=20pt{
&0\ar[d]&0\ar@{-->}[d]&&\\
0\ar[r]&M_{-1}\ar[r]\ar[d]&Q_{-1}\ar[r]\ar@{-->}[d]&G_{-1}^{'}\ar@{==}[d]\ar[r]&0\\
0\ar@{-->}[r]&G_{-2}\ar@{-->}[r]\ar[d]&T_{-2}\ar@{-->}[r]\ar@{-->}[d]&G_{-1}^{'}\ar@{-->}[r]&0\\
&M_{-2}\ar@{==}[r]\ar[d]&M_{-2}\ar@{-->}[d]&&\\
&0&0.&& }$$
Since $G_{-1}^{'},G_{-2}\in r\mathcal{G}(\C)$, we have
$T_{-2}\in r\mathcal{G}(\C)$ by Proposition \ref{extension-closed}(1).
Since $M_{-2}\in{^{\perp}\C}$, the middle column in the above diagram is $\Hom_{\A}(-,\C)$-exact
and we get an exact $\Hom_{\A}(-,\C)$-exact sequence
$$0\rightarrow Q_{-1}\rightarrow T_{-2}\rightarrow G_{-3}\rightarrow G_{-4}\rightarrow \cdots$$
in $\A$.  We repeat the argument by replacing $M$ with $Q_{-1}$ to get an exact $\Hom_{\A}(-,\C)$-exact sequence
$$0\rightarrow Q_{-1}\rightarrow C_{-2}\rightarrow Q_{-2}\rightarrow 0$$ in $\A$ with $C_{-2}\in\C$.
Continuing this process, we obtain an exact $\Hom_{\A}(-,\C)$-exact sequence
$$0 \rightarrow M\rightarrow C_{-1}\rightarrow C_{-2}\rightarrow \cdots$$
in $\A$ with all $C_i$ in $\C$. It implies $M\in r\mathcal{G}(\C)$.
\end{proof}

We write $r\mathcal{G}^2(\C)$ to be a subcategory of $\A$ consists of all $A\in \A$ such that there exists an exact $\Hom_\A(-,\C)$-exact sequence
$$\cdots \to G_{1} \to G_{0} \to G^{0} \to G^{1} \to \cdots$$
in $\A$ with all $G_{i},G^{i}$ in $r\mathcal{G}(\C)$ and
$A \cong \Im (G_{0} \to G^{0})$.
The following result is an immediate consequence of Theorem \ref{3.6},
which generalizes \cite[Theorem 2.9]{LHX13}.

\begin{corollary}\label{3.7}
$r\mathcal{G}^2(\C)=r\mathcal{G}(\C)$.
\end{corollary}

The following result shows that under some conditions,
objects in $r\mathcal{G}(\C)$ possess the symmetry just as objects in $\mathcal{G}(\C)$ do.

\begin{corollary}\label{3.8}
 If $r\mathcal{G}(\C)$ has a projective generator $\mathscr{P}$,
then the following statements are equivalent for any $G\in\A$.
\begin{enumerate}
\item[(1)] $G\in r\mathcal{G}(\C)$.
\item[(2)] There exists an exact $\Hom_\A(-,\C)$-exact sequence
$$\cdots \to P_1\to P_0 \to C^0\to C^1\to \cdots$$
in $\A$ with all $P_i$ in $\mathscr{P}$ and $C^i$ in $\mathscr{C}$ such that $G\cong\Im(P_0 \to C^0)$.
\item[(3)] There exists an exact $\Hom_R(-,\C)$-exact sequence
$$\cdots \to W_1\to W_0\to W^0 \to W^1\to \cdots$$
in $\A$ with all $W_i, W^i\in \C\cup\mathscr{P}$ and $G\cong \Im(W_0\to W^0)$.
\end{enumerate}
\end{corollary}

\begin{proof}
The implications $(1)\Leftrightarrow (2)\Rightarrow (3)$ are clear. Since $\C\cup\mathscr{P}\subseteq r\mathcal{G}(\C)$
by assumption, the implication $(3)\Rightarrow (1)$ follows from Corollary \ref{3.7}.
\end{proof}

In the following, we will give some criteria for computing the $\mathcal{G}(\C)$-projective dimension
of an object in $\A$. We need the following lemma.

\begin{lemma}\label{3.9}
If $$0\to K \to G_1 \buildrel {f} \over \longrightarrow G_0 \to A \to 0$$
is an exact sequence in $\A$ with $G_0,G_1\in r\mathcal{G}(\C)$, then
there exists an exact sequence
$$0\to K \to C \to G \to A \to 0$$ in $\A$ with $C\in\C$ and $G\in r\mathcal{G}(\C)$.
\end{lemma}

\begin{proof}
Since $G_1\in r\mathcal{G}(\C)$, by Lemma \ref{3.3} there exists an exact sequence
$$0\to G_1 \to C \to G^{'}\to 0$$
in $\A$ with $C\in\C$ and
$G^{'}\in r\mathcal{G}(\C)$.
The following push-out diagram
$$\xymatrix{
& & 0 \ar[d] &0 \ar@{-->}[d] &   \\
0 \ar[r]  & K \ar@{==}[d]  \ar[r] & G_1 \ar[d]
\ar[r]  &\Im f \ar@{-->}[d] \ar[r] & 0 \\
0 \ar@{-->}[r] & K \ar@{-->}[r] & C \ar[d] \ar@{-->}[r] & L \ar@{-->}[d] \ar@{-->}[r] & 0  \\
&  & G^{'} \ar[d]
\ar@{==}[r] & G^{'}\ar@{-->}[d]  &   \\
&  & 0  & 0  &   } $$
yields a push-out diagram
$$\xymatrix{
& 0 \ar[d] &0 \ar@{-->}[d] &   &    \\
0 \ar[r] & \Im f  \ar[d] \ar[r] & G_0
\ar@{-->}[d]  \ar[r]  & A \ar@{==}[d] \ar[r]  &0  \\
0 \ar@{-->}[r]  & L \ar[d]  \ar@{-->}[r] &G
\ar@{-->}[d] \ar@{-->}[r] & A \ar@{-->}[r]  &0  \\
& G^{'}\ar[d]\ar@{==}[r] &G^{'} \ar@{-->}[d] &   &    \\
&  0 & 0.  &  & }$$
By Proposition \ref{extension-closed}(1), the middle column in the last diagram implies that $G\in r\mathcal{G}(\C)$.
Splicing the middle rows of the two diagrams above, we get the desired exact sequence.
\end{proof}

The following result provides some criteria for computing the $\mathcal{G}(\C)$-projective dimension
of an object in $\A$, which shows that any object in $\A$ with finite $r\mathcal{G}(\C)$-projective dimension
is isomorphic to a kernel (resp. a cokernel) of a morphism from an object in $\A$
with finite $\C$-projective dimension to an object in $r\mathcal{G}(\C)$.

\begin{theorem}\label{3.10} For any $A\in \A$ and $n\geq 0$, the following statements are equivalent.
\begin{enumerate}
\item[(1)] $r\mathcal{G}(\C)$-$\pd A\leq n$.
\item[(2)] There exists an exact sequence
$$0\to H \to G \to A\to 0$$
in $\A$ with $G\in r\mathcal{G}(\C)$ and $\C$-$\pd H\leq n-1$.
\item[(3)] There exists an exact ($\Hom_{\A}(-,\C)$-exact) sequence
$$0\to A \to H^{'} \to G^{'} \to 0$$
in $\A$ with $G^{'}\in r\mathcal{G}(\C)$ and $\C$-$\pd H^{'}\leq n$.
\end{enumerate}
\end{theorem}

\begin{proof}
$(1)\Rightarrow (2)$
We proceed by induction on $n$. If $n=0$ then $H=0$
and $G=A$ give the desired exact sequence. If $n=1$ then there
exists an exact sequence
$$0\to G_1\to G_0\to A\to 0$$
in $\A$ with $G_0,G_1\in r\mathcal{G}(\C)$. Applying Lemma \ref{3.9} with $K=0$,
we get an exact sequence
$$0\to C\to G^{'}_0\to A\to 0$$
in $\A$ with $C\in\C$ and $G^{'}_0\in r\mathcal{G}(\C)$.
Now suppose $n\geq 2$. Then there exists an exact sequence
$$0\to G_n\to G_{n-1}\to\cdots \to G_0\to A\to 0$$
in $\A$ with all $G_i$ in $r\mathcal{G}(\C)$.
Set $T:=\Im(G_{1}\to G_{0})$. By the induction hypothesis, we get the following exact sequence
$$0\to C_n\to C_{n-1}\to C_{n-2}\to \cdots \to C_{2}\to G^{'}_1\to T\to 0$$
in $\A$ with all $C_i$ in $\C$ and $G^{'}_1\in r\mathcal{G}(\C)$.
Set $B:=\Im(C_2\to G^{'}_1)$. By Lemma \ref{3.9}, we get an exact sequence
$$0\to B\to C_1\to G\to A\to 0$$
in $\A$ with $C_1\in \C$ and $G\in r\mathcal{G}(\C)$. Thus we get
the desired exact sequence
$$0\to C_n\to C_{n-1}\to C_{n-2}\to\cdots \to C_{1}\to G\to A\to 0.$$
Then we get the desired exact sequence by putting $H:=\Coker(C_2\to C_1)$.

$(2)\Rightarrow (3)$ Let
$$0\to H \to G \to A\to 0$$
be an exact sequence in $\A$ with $G\in r\mathcal{G}(\C)$
and $\C$-$\pd H\leq n-1$. Since $G\in r\mathcal{G}(\C)$, by Lemma \ref{3.3} there exists an exact $\Hom_{\A}(-,\C)$-exact sequence
$$0\to G\to C\to G^{'}\to 0$$ in $\A$ with $C\in\C$
and
$G^{'}\in r\mathcal{G}(\C)$. Consider the following push-out diagram
$$\xymatrix{
& & 0 \ar[d] &0 \ar@{-->}[d] &   \\
0 \ar[r]  & H \ar@{==}[d]  \ar[r] & G \ar[d]\ar[r]  &A \ar@{-->}[d] \ar[r] & 0 \\
0 \ar@{-->}[r] & H \ar@{-->}[r] & C \ar[d] \ar@{-->}[r] & H^{'} \ar@{-->}[d] \ar@{-->}[r] & 0  \\
&  & G^{'} \ar[d]
\ar@{==}[r] & G^{'}\ar@{-->}[d]  &   \\
&  & 0  & 0.   &   } $$
By the middle row in this diagram, we have $\C$-$\pd H^{'}\leq n$. Moreover, since $G^{'}\in r\mathcal{G}(\C)$, the rightmost column is $\Hom_{\A}(-,\C)$-exact.
Thus we get the desired exact sequence.

$(3)\Rightarrow (1)$ Let
$$0\to A \to H^{'} \to G^{'} \to 0$$
be an exact sequence in $\A$ with $G^{'}\in r\mathcal{G}(\C)$ and $\C$-$\pd H^{'}\leq n$.
Then there exists an exact sequence
$$0\ra C_n\ra \cdots\ra C_1\ra C_0\ra H^{'}\ra 0$$ in $\A$ with all $C_i$ in $\C$. Set $K:=\mbox{Ker}(C_0\ra H^{'})$.
Then $\C$-$\pd K\leq n-1$. Consider the following pull-back diagram
$$\xymatrix@R=20pt@C=20pt{& 0 \ar@{-->}[d] & 0 \ar[d]&& &\\
& K \ar@{==}[r] \ar@{-->}[d] & K \ar[d]& &&\\
0 \ar@{-->}[r] & G \ar@{-->}[d] \ar@{-->}[r] & C_0 \ar[d] \ar@{-->}[r] &G^{'} \ar@{==}[d] \ar@{-->}[r] & 0\\
0 \ar[r] & A \ar@{-->}[d]\ar[r] & H^{'} \ar[r] \ar[d] & G^{'} \ar[r] & 0 &\\
& 0 & 0. & && }$$
Applying Proposition \ref{extension-closed}(2) to the middle row in this diagram yields  $G\in r\mathcal{G}(\C)$.
Thus $r\mathcal{G}(\C)$-$\pd A\leq n$ by the leftmost column in the above diagram.
\end{proof}

\section{Weak Auslander-Buchweitz contexts}

Following the approximation theory of Auslander-Buchweitz \cite{AuBu}, Hashimoto \cite{Ha00} introduced
the following definition.

\begin{definition}\label{4.1} (\cite{Ha00})
A triple $(\mathscr{X},\mathscr{Y},\omega)$ of subcategories of $\A$ is called
a {\it weak Auslander-Buchweitz context} in $\A$ if the following conditions are satisfied.
\begin{enumerate}
\item[(1)] $\mathscr{X}$ is closed under extensions, kernels of epimorphisms and direct summands.
\item[(2)] $\mathscr{Y}\subseteq \mathscr{X}$-$\pd^{<\infty}$ and $\mathscr{Y}$ is closed under extensions,
cokernels of monomorphisms and direct summands.
\item[(3)] $\omega=\mathscr{X}\cap \mathscr{Y}$ and $\omega$ is an injective cogenerator for $\mathscr{X}$.
\end{enumerate}
A weak Auslander-Buchweitz context $(\mathscr{X},\mathscr{Y},\omega)$ is called an {\it Auslander-Buchweitz context}
if $\A=\mathscr{X}$-$\pd^{<\infty}$.
\end{definition}

\begin{definition}\label{4.2}  (\cite{EO02})
Let $\mathscr{U},\mathscr{V}$ be subcategories of $\A$.
\begin{enumerate}
\item[(1)] The pair $(\mathscr{U},\mathscr{V})$ is called a {\it cotorsion pair} in $\A$ if
$\mathscr{U}={^{\bot_1}\mathscr{V}}$ and $\mathscr{V}={\mathscr{U}^{\bot_1}}$; in this case,
$\C:=\mathscr{U}\cap\mathscr{V}$ is called the {\it kernel} of $(\mathscr{U},\mathscr{V})$.
\item[(2)] A cotorsion pair $(\mathscr{U},\mathscr{V})$ is said to {\it have enough injectives} if
for any $A\in\A$, there exists an exact sequence
$$0\to A \to V \to U \to 0$$
in $\A$ with $V\in\mathscr{V}$ and $U\in\mathscr{U}$. Dually, a cotorsion pair
$(\mathscr{U},\mathscr{V})$ is said to {\it have enough projectives} if
for any $A\in\A$, there exists an exact sequence
$$0 \to V^{'} \to U^{'} \to A \to 0$$
in $\A$ with $V^{'}\in\mathscr{V}$ and $U^{'}\in\mathscr{U}$.
\item[(3)] A cotorsion pair $(\mathscr{U},\mathscr{V})$ is called {\it hereditary} if one of the following equivalent
conditions is satisfied.
\begin{enumerate}
\item[(3.1)] $\mathscr{U}\perp \mathscr{V}$.
\item[(3.2)] $\mathscr{U}$ is resolving in the sense that $\mathscr{U}$ contains all projectives
in $\A$, $\mathscr{U}$ is closed under extensions and kernels of epimorphisms.
\item[(3.3)] $\mathscr{V}$ is coresolving in the sense that $\mathscr{V}$ contains all injectives
in $\A$, $\mathscr{V}$ is closed under extensions and cokernels of monomorphisms.
\end{enumerate}
\end{enumerate}
\end{definition}

In what follows, $(\mathscr{U},\mathscr{V})$ is a hereditary cotorsion pair and $\C:=\mathscr{U}\cap \mathscr{V}$ is its kernel.
The following is a standard observation.

\begin{lemma}\label{4.3}
For any $A\in\A$ and $n\geq 0$, the following statements are equivalent.
\begin{enumerate}
\item[(1)] $\mathscr{U}\mbox{-}\operatorname{pd}A\leq n$.
\item[(2)] $\Ext^{n+1}_{\A}(A,V)=0$ for any $V\in\mathscr{V}$.
\item[(3)] $\Ext^{\geq n+1}_{\A}(A,V)=0$ for any $V\in\mathscr{V}$.
\end{enumerate}
\end{lemma}

\begin{lemma}\label{4.4}
If $(\mathscr{U},\mathscr{V})$ has enough injectives, then $\C$ is an injective cogenerator for $\mathscr{U}$.
\end{lemma}

\begin{proof}
If $(\mathscr{U},\mathscr{V})$ has enough injectives, then for any $U\in\mathscr{U}$ there exists an exact sequence
$$0\to U \to V \to U^{1} \to 0$$
in $\A$ with $V\in\mathscr{V}$ and $U^{1}\in\mathscr{U}$. Since $\mathscr{U}$ is closed under extensions,
we have $V\in\mathscr{U}$ and so $V\in\C$. Since $\mathscr{U}\bot\C$ clearly,
$\C$ is an injective cogenerator for $\mathscr{U}$.
\end{proof}

The following result is used for obtaining our main result (Theorem \ref{4.8}).

\begin{proposition}\label{4.5}
If $(\mathscr{U},\mathscr{V})$ has enough injectives, then we have
\begin{enumerate}
\item[(1)] $r\mathcal{G}(\C)\cap \mathscr{U}$-$\pd^{<\infty}=\mathscr{U}$.
\item[(2)] $\C$-$\pd^{<\infty}=\mathscr{U}$-$\pd^{<\infty}\cap \mathscr{V}$.
\item[(3)] $\C$-$\pd^{<\infty}$ is closed under extensions, cokernels of monomorphisms and direct summands.
\end{enumerate}
\end{proposition}

\begin{proof}
(1) It is easy to see that $\mathscr{U}\subseteq r\mathcal{G}(\C)\cap \mathscr{U}$-$\pd^{<\infty}$
by Lemma \ref{4.4}.

Conversely, let $A\in r\mathcal{G}(\C)\cap \mathscr{U}$-$\pd^{<\infty}$. We will prove $A\in \mathscr{U}$
by induction on $n:=\mathscr{U}$-$\pd A$.
The case for $n=0$ is trivial. Now let $n\geq 1$ and
$$0\ra L\ra U_0\ra A\ra 0$$
be an exact sequence in $\A$ with $U_0\in\mathscr{U}(\subseteq r\mathcal{G}(\C))$
and $\mathscr{U}$-$\pd L\leq n-1$. By Proposition \ref{extension-closed}(2), we have $L\in r\mathcal{G}(\C)$.
Thus $L\in\mathscr{U}$ by the induction hypothesis. By Lemma \ref{4.4}, we have an exact sequence
$$0\ra L\ra C\ra U\ra 0$$ in $\A$ with $C\in \C$  and $U\in\mathscr{U}$. Consider the following push-out diagram
$$\xymatrix@R=20pt@C=20pt{
&0\ar[d]&0\ar@{-->}[d]&&\\
0\ar[r]&L\ar[r]\ar[d]&U_0\ar[r]\ar@{-->}[d]&A\ar@{==}[d]\ar[r]&0\\
0\ar@{-->}[r]&C\ar@{-->}[r]\ar[d]&Q\ar@{-->}[r]\ar@{-->}[d]&A\ar@{-->}[r]&0\\
&U\ar@{==}[r]\ar[d]&U\ar@{-->}[d]&&\\
&0&0.&& }$$
By the middle column in this diagram, we have $Q\in \mathscr{U}$. Since $A\in r\mathcal{G}(\C)$ and $C\in\C$,
it follows that $\Ext^1_{\A}(A,C)=0$ and the middle row in the above diagram splits.
Thus $A$ is isomorphic to a direct summand of $Q$ and $A\in\mathscr{U}$.

(2) Clearly, $\C$-$\pd^{<\infty}\subseteq\mathscr{U}$-$\pd^{<\infty}$.
Since $\mathscr{V}$ is closed under cokernels of monomorphisms,
we have $\C$-$\pd^{<\infty}\subseteq \mathscr{V}$. Thus
$\C$-$\pd^{<\infty}\subseteq\mathscr{U}$-$\pd^{<\infty}\cap \mathscr{V}$.

Conversely, let $A\in\mathscr{U}$-$\pd^{<\infty}\cap \mathscr{V}$ and $\mathscr{U}$-$\pd A=n(<\infty)$.
Since $\mathscr{U}\subseteq r\mathcal{G}(\C)$, by Theorem \ref{3.10} there exists an exact sequence
$$0\to H \to G \to A\to 0$$
in $\A$ with $G\in r\mathcal{G}(\C)$ and $\C$-$\pd H\leq n-1$. By Lemma \ref{4.3}, we have
$\mathscr{U}$-$\pd G\leq n$. Then $G\in\mathscr{U}$ by (1).
Since $\mathscr{V}$ is closed under cokernels of monomorphisms and extensions, we have $H\in\mathscr{V}$,
and hence $G\in\mathscr{V}$. Thus $G\in\mathscr{C}(=\mathscr{U}\cap \mathscr{V})$ and
$\mathscr{C}\mbox{-}\operatorname{pd}A\leq n$.

(3) By Lemma \ref{4.3}, $\mathscr{U}$-$\pd^{<\infty}$ is closed under direct summands. Since
$\mathscr{V}$ is closed under direct summands, so is $\C$-$\pd^{<\infty}$ by (2).

Since $\C\bot\C$, we have $\C$-$\pd^{<\infty}\subseteq{\C^{\bot}}$ by the dimension shifting.
Then any short exact sequence in $\A$ with the first term in $\C$-$\pd^{<\infty}$ is $\Hom_\A(\C,-)$-exact.
Therefore by \cite[Lemma 3.1(1)]{Hu}, we get $\C$-$\pd^{<\infty}$ is closed under extensions, and by \cite[Theorem 3.6(1)]{Hu}, we get
$\C$-$\pd^{<\infty}$ is closed under cokernels of monomorphisms.
\end{proof}

As a consequence, we have the following proposition.

\begin{proposition}\label{4.6}
Assume that $(\mathscr{U},\mathscr{V})$ has enough injectives. If $(r\mathcal{G}(\C),\C$-$\pd^{<\infty},\C)$
is an Auslander-Buchweitz context, then
$$(r\mathcal{G}(\C),\mathscr{U}\text{-}\pd^{<\infty}\cap \mathscr{V})=(r\mathcal{G}(\C),\C\text{-}\pd^{<\infty})$$
is a cotorsion pair.
\end{proposition}

\begin{proof}
By Proposition \ref{4.5}(2),
$\mathscr{U}$-$\pd^{<\infty}\cap \mathscr{V}=\C$-$\pd^{<\infty}$.
It is easy to see that $r\mathcal{G}(\C)\bot\C$-$\pd^{<\infty}$ by the dimension shifting.
Thus $r\mathcal{G}(\C)\subseteq{^{\bot_1}}(\C$-$\pd^{<\infty})$ and
$\C$-$\pd^{<\infty}\subseteq {r\mathcal{G}(\C)^{\bot_1}}$.
Since $(r\mathcal{G}(\C),\C$-$\pd^{<\infty},\C)$
is an Auslander-Buchweitz context by assumption, we have $\A=r\mathcal{G}(\mathscr{C})$-$\pd^{<\infty}$.

Let $A\in {^{\bot_1}}(\C$-$\pd^{<\infty})$. Then
$r\mathcal{G}(\mathscr{C})$-$\pd A<\infty$. By Theorem \ref{3.10}, there exists an exact sequence
$$0\to H \to G \to A\to 0$$
in $\A$ with $G\in r\mathcal{G}(\C)$ and $\C$-$\pd H<\infty$. Thus this exact sequence splits,
and hence $A$ is isomorphic to a direct summand of $G$. Then by Proposition \ref{extension-closed}(1),
we have $A\in r\mathcal{G}(\C)$ and ${^{\bot_1}}(\C$-$\pd^{<\infty})\subseteq r\mathcal{G}(\C)$.

Now let $A\in {r\mathcal{G}(\C)^{\bot_1}}$. Note that $r\mathcal{G}(\mathscr{C})$-$\pd A<\infty$.
By Theorem \ref{3.10}, there exists an exact sequence
$$0\to A \to H^{'} \to G^{'}\to 0$$
in $\A$ with $G^{'}\in r\mathcal{G}(\C)$ and $\C$-$\pd H^{'}<\infty$. Thus this exact sequence splits,
and hence $A$ is isomorphic to a direct summand of $H^{'}$. Then by Proposition \ref{4.5}(3),
we have $\C$-$\pd A<\infty$ and ${r\mathcal{G}(\C)^{\bot_1}}\subseteq\C$-$\pd^{<\infty}$.
\end{proof}

We also need the following easy observation.

\begin{lemma}\label{4.7}
\begin{enumerate}
\item[]
\item[(1)] $\C$ is an injective cogenerator for $r\mathcal{G}(\C)$.
\item[(2)] $r\mathcal{G}(\C)\cap \C$-$\pd^{<\infty}=\C$.
\end{enumerate}
\end{lemma}

\begin{proof}
The assertion (1) is trivial.
Since $\C\bot\C$, we have $^{\bot}\C\cap \C$-$\pd^{<\infty}=\C$ by the dimension shifting,
so the assertion (2) follows.
\end{proof}

Now we are in a position to give the following theorem.

\begin{theorem}\label{4.8}
If $(\mathscr{U},\mathscr{V})$ has enough injectives, then
$(r\mathcal{G}(\C),\C$-$\pd^{<\infty},\C)$ is a weak Auslander-Buchweitz context.
\end{theorem}

\begin{proof}
It follows from Propositions \ref{extension-closed}  and \ref{4.5}(3), and Lemma \ref{4.7}.
\end{proof}

By Theorem \ref{4.8} and \cite[Theorem 1.12.10]{Ha00}, we immediately get the following corollary.

\begin{corollary}\label{4.9}
If $(\mathscr{U},\mathscr{V})$ has enough injectives, then we have
\begin{enumerate}
\item[(1)] ${\C}$ is a unique additive injective cogenerator for $r\mathcal{G}(\C)$ in the sense that
if $\mathscr{E}$ is an injective cogenerator for $r\mathcal{G}(\C)$, then $\operatorname{add} \mathscr{E}={\C}$,
where $\operatorname{add} \mathscr{E}$ is the subcategory of $\A$ consisting of direct summands of finite
direct sums of objects in $\mathscr{E}$.
\item[(2)] For any $A\in r\mathcal{G}(\C)$-$\pd^{<\infty}$, the following statements are equivalent.

~~~~$\mathrm{(i)}$ $A\in r\mathcal{G}(\C)$;
~~~~$\mathrm{(ii)}$ $A\in{^\perp(\C}$-$\pd^{<\infty})$;
~~~~$\mathrm{(iii)}$ $A\in{^{\perp_1}(\C}$-$\pd^{<\infty})$;
~~~~$\mathrm{(iv)}$ $A\in {^\perp\C}$.
\item[(3)] For any $A\in r\mathcal{G}(\C)$-$\pd^{<\infty}$, the following statements are equivalent.

~~$\mathrm{(i)}$ $A\in \C$-$\pd^{<\infty}$;
~~$\mathrm{(ii)}$ $A\in{r\mathcal{G}(\C)^\perp}$;
~~$\mathrm{(iii)}$ $A\in{r\mathcal{G}(\C)^{\perp_1}}$;
~~$\mathrm{(iv)}$ $r\mathcal{G}(\C)\mbox{-}\id A <\infty$ and $A\in {\C^\perp}$.
\item[(4)] For any $A\in r\mathcal{G}(\C)$-$\pd^{<\infty}$, we have
\begin{align*}
r\mathcal{G}(\C)\text{-}\pd A & =\operatorname{inf}\{n\mid \Ext^{n+1}_\A(A,C)=0 \mbox{ for any }C\in \C)\} \\
   & =\operatorname{inf}\{n\mid \Ext^{n+1}_\A(A,W)=0 \mbox{ for any }W\in \C\text{-}\pd^{<\infty}\}.
\end{align*}
\item[(5)] For any $A\in\C\text{-}\pd^{<\infty}$, we have $r\mathcal{G}(\C)\text{-}\pd A=\C\text{-}\pd A$.
\item[(6)] For an exact sequence
$$0\rightarrow L\rightarrow M\rightarrow N\rightarrow 0$$ in $\A$, if any two of $L$, $M$ and $N$ are in
$r\mathcal{G}(\C)$-$\pd^{<\infty}$, then so is the third one.
    \end{enumerate}
\end{corollary}

As a dual of Definition \ref{4.1}, we give the following definition.

\begin{definition}\label{4.10}
 A triple $(\mathscr{X},\mathscr{Y},\omega)$ of subcategories of $\mathcal{A}$ is called
a {\it weak co-Auslander-Buchweitz context} if the following conditions are satisfied.
\begin{enumerate}
\item[(1)] $\mathscr{X}$ is closed under extensions, cokernels of monomorphisms and direct summands.
\item[(2)] $\mathscr{Y}\subseteq \mathscr{X}$-$\id^{<\infty}$ and $\mathscr{Y}$ is
closed under kernels of epimorphisms, extensions and direct summands.
\item[(3)] $\omega=\mathscr{X}\cap \mathscr{Y}$ and $\omega$ is a projective generator for $\mathscr{X}$.
\end{enumerate}
A weak co-Auslander-Buchweitz context $(\mathscr{X},\mathscr{Y},\omega)$ is called a {\it co-Auslander-Buchweitz context}
if $\mathscr{X}$-$\id^{<\infty}=\mathcal{A}$.
\end{definition}

By using arguments completely dual to that in the proofs of Proposition \ref{4.6} and Theorem \ref{4.8}, we have the following proposition.

\begin{proposition}\label{4.11}
Assume that $(\mathscr{U},\mathscr{V})$ has enough projectives. If $(l\mathcal{G}(\C), \mathscr{C}$-$\id^{<\infty},\C)$
is a co-Auslander-Buchweitz context, then
$$(l\mathcal{G}(\C),\mathscr{U}\cap \mathscr{V}\text{-}\id^{<\infty})=(r\mathcal{G}(\C), \mathscr{C}\text{-}\id^{<\infty})$$
is a cotorsion pair.
\end{proposition}

\begin{theorem}\label{4.12}
If $(\mathscr{U},\mathscr{V})$ has enough projectives,
then $(l\mathcal{G}(\C), \mathscr{C}$-$\id^{<\infty},\C)$ is a weak co-Auslander-Buchweitz context.
\end{theorem}

\vspace{0.5cm}

{\bf Acknowledgements.}
This research was partially supported by NSFC (Grant No. 11571164).
The authors thank the referee for very useful and detailed suggestions.

\end{document}